\documentclass[12pt]{amsart}
\usepackage{amssymb,amsthm,amsmath,epsfig,latexsym,calc,tikz}
\usepackage{verbatim}

\numberwithin{equation}{section}
\newtheorem{theorem}{Theorem}[section]
\newtheorem{lemma}[theorem]{Lemma}

\newtheorem{corollary}[theorem]{Corollary}
\newtheorem{conjecture}[theorem]{Conjecture}

\theoremstyle{definition}
\newtheorem{definition}[theorem]{Definition}
\newtheorem{construction}[theorem]{Construction}
\newtheorem{remark}[theorem]{Remark}
\newtheorem{example}[theorem]{Example}

\begin{document}

\title{Balanced vertex decomposable simplicial complexes and their
$h$-vectors}

\thanks{Version: May 9, 2012}

\author{Jennifer Biermann}
\address{Department of Mathematical Sciences, 
Lakehead University, 
Thunder Bay, ON P7B 5E1, Canada}
\email{jvbierma@lakeheadu.ca}
\urladdr{http://jenniferbiermann.lakeheadu.ca/}

\author{Adam Van Tuyl}
\address{Department of Mathematical Sciences, 
Lakehead University, 
Thunder Bay, ON P7B 5E1, Canada}
\email{avantuyl@lakeheadu.ca}
\urladdr{http://flash.lakeheadu.ca/~avantuyl}


\begin{abstract}
Given any finite simplicial complex $\Delta$, we show how to construct
a new simplicial complex $\Delta_{\chi}$ that is balanced
and vertex decomposable.  Moreover, we show that the
$h$-vector of the simplicial complex  $\Delta_{\chi}$ is precisely
the $f$-vector of the original complex $\Delta$.  
Our construction generalizes the ``whiskering'' construction
of Villarreal, and Cook and Nagel.  As a corollary of our work,
we add a new equivalent statement to a theorem of Bj\"orner,
Frankl, and Stanley that classifies the $f$-vectors of simplicial
complexes.  We also prove a special case of a conjecture
of Cook and Nagel, and  Constantinescu and Varbaro on the 
$h$-vectors of flag complexes.
\end{abstract}

\maketitle


\section{Introduction}

The work of this paper was inspired by the ``whiskering'' construction
of finite simple graphs
found in work of
Villarreal \cite{V} and Cook and Nagel \cite{CN}.  Given a finite
graph $G = (V_G,E_G)$ on the vertex set $V_G = \{x_1,\ldots,x_n\}$,
Villarreal constructed a new graph, denoted $G^W$, on the vertex set 
$\{x_1,\ldots,x_n,y_1,\ldots,y_n\}$ by adjoining the edges $\{x_i,y_i\}$
for every $i$ to the graph $G$.  The new graph has a ``whisker'' at every
vertex of the original graph.
As discovered by Villarreal,
the edge ideal of the new graph $G^W$, that is, 
\[I(G^W) = 
\langle w_iw_j ~
~|~ \{w_i,w_j\} \in E_{G^W} \rangle \subseteq R =
k[x_1,\ldots,x_n,y_1,\ldots,y_n]\] has the property that
$R/I(G^W)$ is Cohen-Macaulay.  It was later observed by
Dochtermann and Engstr\"om \cite{DE} and Woodroofe \cite{W} and generalized by
Cook and Nagel \cite{CN}, that one could deduce this result
by studying
the topological properties of the
simplicial complex associated to $I(G^W)$ via the Stanley-Reisner
correspondence.  In particular, Villarreal's construction can be viewed
as creating a new independence complex $\Delta'$ 
(sometimes called a flag complex)
from the independence complex $\Delta$ of $G$.   This new complex
$\Delta'$ is vertex decomposable (as defined by Provan and Billera \cite{BP}),
and it is this topological property that implies that $R/I(G^W)$
is Cohen-Macaulay.

Our entry point was to ask  whether there is a more general 
theory that can be
applied to all simplicial complexes. Moreover, we want this general
theory to specialize to known cases for flag complexes.  We will
show that a general construction exists using the notion
of a colouring $\chi$ of a simplicial complex $\Delta$ (all
terms will be properly defined in the following sections).   
From the colouring $\chi$ and complex $\Delta$, we make a new complex,
denoted $\Delta_{\chi}$.

The first main result of this paper is to show that regardless
of how one colours $\Delta$, the construction of $\Delta_{\chi}$ always results
in a vertex decomposable simplicial complex:

\begin{theorem}[Theorem \ref{vertex decomposable}] 
For any simplicial complex $\Delta$, and any $s$-colouring $\chi$ of $\Delta$, 
the simplicial complex $\Delta_\chi$ is balanced and vertex decomposable.
\end{theorem}

\noindent
Here, balanced means the simplicial complex has a colouring with $(\dim \Delta_{\chi})+1$ 
colours.
Results of \cite{CN,DE,V} now become
special cases of this theorem since ``whiskering'' will be shown to be
equivalent to colouring the independence complex of a graph.

We investigate the consequences of Theorem \ref{vertex decomposable}
in Section 4.  One such consequence 
is the addition of the implication $(i) \implies (ii)$ 
to the following theorem:

\begin{theorem}[Theorem \ref{BFSThem}] Let $m = (m_1,\ldots,m_t) \in \mathbb{Z}_+^t$.  The following are equivalent:
\begin{enumerate}
\item[$(i)$] $m$ is the $f$-vector of a simplicial complex.
\item[$(ii)$] $m$ is the $h$-vector of a balanced, vertex decomposable simplicial complex.
\item[$(iii)$] $m$ is the $h$-vector of a balanced, shellable simplicial complex.
\item[$(iv)$] $m$ is the $h$-vector of a balanced, Cohen-Macaulay simplicial complex.
\end{enumerate}
\end{theorem}
The equivalence of statements $(i), (iii)$ and $(iv)$ was first proved by
Bj\"orner, Frankl, and Stanley \cite{BFS}. 
It should be noted that versions of $(i) \implies (ii)$ have appeared in the
literature in special cases (see, e.g., \cite[Proposition 4.1]{CV},\cite[Proposition 3.8]{CN}, 
\cite[Proposition 3.7]{F}), but to the best of our knowledge, no version
of the above theorem has appeared before.

Another consequence is a formula for the graded Betti
numbers of the Stanley-Reisner ideal of the Alexander dual of $\Delta_{\chi}$ in terms
of the $f$-vector of $\Delta$.  

\begin{theorem}[Theorem \ref{bettinumbers}]
Let $f(\Delta) = (f_{-1},f_0,\ldots,f_{d})$ be the $f$-vector
of a $d$-dimensional simplicial complex $\Delta$
on $V= \{x_1,\ldots,x_n\}$, and let $\chi$
be any $s$-colouring of $\Delta$.   Then, for all $i \geq 0$,
\[\beta_{i,n+i}(I_{\Delta_{\chi}^{\vee}}) = \sum_{j=i}^{d+1} \binom{j}{i} f_{j-1}(\Delta).\]
\end{theorem}

\noindent
Because $\Delta_{\chi}$ is vertex decomposable, $R/I_{\Delta_{\chi}}$ is also
Cohen-Macaulay, so by the Eagon-Reiner Theorem \cite{ER}, 
the ideal $I_{\Delta_{\chi}^\vee}$ has 
a linear resolution.  Thus Theorem \ref{bettinumbers} describes all the Betti
numbers of $I_{\Delta^\vee_{\chi}}$.  Thus, starting from
any $f$-vector, we can construct an ideal with a linear resolution whose
Betti numbers only depend upon the $f$-vector.  
This result could also be deduced from recent work
 Herzog, Sharifan, and Varbaro \cite{HSV} which classifies all
sequences which can be the sequence of Betti numbers for an ideal with a linear 
minimal free resolution. 
However, the ideals of \cite{HSV} need not be square-free monomial ideals.

We round out this paper by describing when our construction can be reversed
so that one can start with a balanced vertex decomposable simplicial complex $\Delta$
and construct another simplicial complex $\Delta'$ such that $f$-vector
of $\Delta'$ is the same as the $h$-vector of $\Delta$.  We use this procedure to prove:
\[\left\{
\begin{array}{c}
\mbox{$f$-vectors of independence} \\
\mbox{complexes of chordal graphs}
\end{array}
\right\} =
\left\{
\begin{array}{c}
\mbox{$h$-vectors of balanced vertex decomposable}\\
\mbox{independence complexes of chordal graphs}
\end{array}
\right\}.\]
This proves a special case of a conjecture of Cook and Nagel \cite{CN} and
Constantinescu and Varbaro \cite{CV} that the set of $f$-vectors of a flag
complexes is precisely the set of $h$-vectors of balanced vertex 
decomposable flag complexes.  

As a final comment, this paper does not discuss the ``whiskering'' procedure 
found in \cite{FH} in which
whiskers are added to only some of the vertices.  In ongoing work with Francisco and H\`a,
we are currently investigating how to partially whisker a simplicial complex.

\noindent
{\bf Acknowledgements.} 
The authors made use of the computer
programs {\tt CoCoA} \cite{C} and {\em Macaulay 2} \cite{Mt}, including the {\em Macaulay 2}
package {\tt SimplicialDecomposability} of David Cook II \cite{Ct}.
The second author acknowledges the support of NSERC.


\section{Prerequisite background on simplicial complexes}

We work over the polynomial rings 
$S = k[x_1, \dots, x_n]$ and $R= k[x_1, \dots, x_n, y_1, \dots, y_s]$ 
where $k$ is any field.  We recall the relevant background
on simplicial complexes.

\begin{definition} A finite {\it simplicial complex} $\Delta$ on a
finite vertex set $V$ is a 
collection of subsets of $V$ with the 
property that if $\sigma \in \Delta$ and $\tau$ is a subset of $\sigma$, 
then $\tau \in \Delta$.  The elements of $\Delta$ are called {\it faces}.
\end{definition}

The vertex sets of our simplicial complexes will 
be either the set $\{x_1, \dots, x_n\}$ 
or $\{x_1, \dots, x_n, y_1, \dots, y_s\}$.  Because of this, we sometimes write faces as monomials.

If $\Delta$ is a simplicial complex and $\sigma \in \Delta$, then
we say $\sigma$ has {\it dimension} $d$ if the $|\sigma| = d+1$ 
(by convention, the empty set has dimension -1).  The 
maximal faces of $\Delta$ with respect to inclusion are called
 the {\it facets} of $\Delta$.   The {\it dimension} of $\Delta$ is 
the maximum of the dimensions of its facets.  If all of the facets of $\Delta$ are of the same dimension we say that $\Delta$ is {\em pure}.  
If $F_1,\ldots,F_t$ is a complete
list of the facets of $\Delta$, we sometimes write $\Delta$ as
$\Delta = \langle F_1,\ldots,F_t\rangle$.

An important combinatorial invariant of a simplicial complex is its 
$f$-vector.

\begin{definition} Let $\Delta$ be a finite simplicial complex of dimension $d$ 
and let $f_i$ denote the number of faces of $\Delta$ of dimension $i$.  The 
{\it $f$-vector} of $\Delta$, denoted $f(\Delta)$, is then the vector
\[
f(\Delta) = (f_{-1}, f_0, \dots, f_d) .
\]
\end{definition}

We now recall some important operations on simplicial complexes.
If $\sigma$ is a face of a simplicial complex $\Delta$, then the 
{\it deletion} of $\sigma$ from $\Delta$ is the simplicial complex defined by
\[
\Delta \setminus \sigma = \{ \tau \in \Delta ~|~ \sigma \not \subseteq \tau\} .
\]
The {\it link} of $\sigma$ in $\Delta$ is the simplicial complex defined by
\[
\rm{link}_{\Delta}(\sigma) = \{ \tau \in \Delta ~|~ \sigma \cap \tau = 
\emptyset, \sigma \cup \tau \in \Delta\}.
\]
When $\sigma = \{v\}$, we shall abuse notation and
write $\Delta \setminus v$ (respectively ${\rm link}_{\Delta}(v)$) for
$\Delta \setminus \{v\}$ (respectively ${\rm link}_{\Delta}(\{v\})$).

We shall be particularly interested in the following class of
simplicial complexes.  This class was first introduced in the pure case by Provan
and Billera \cite{BP}.

\begin{definition} A pure simplicial complex $\Delta$ 
is called \emph{vertex decomposable} if
\begin{enumerate}
\item $\Delta$ is a simplex, or 
\item there is some vertex $v \in V$ such that $\Delta \setminus v$ and 
${\rm link}_\Delta(v)$ are vertex decomposable.
\end{enumerate}
\end{definition}
Although there is a notion of non-pure vertex decomposability (see
\cite{BW}), in this paper we assume that all vertex decomposable simplicial complexes are pure.

Key to our main construction introduced in Section 3 is the notion
of a colouring.

\begin{definition} Let $\Delta$ be a simplicial complex on the vertex set $V$ 
 with facets $F_1, \dots, F_t$. 
An {\it s-colouring} of $\Delta$  is a partition of the vertices 
 $V = V_1\cup \dots \cup V_s$ (where the sets $V_i$ are allowed to be empty) such that 
$|F_i \cap V_j| \leq 1$
for all 
$1 \leq i \leq t, 1 \leq j \leq s$.
We will sometimes write $\chi$ is an $s$-colouring of $\Delta$
to mean $\chi$ is a specific partition of $V$ that gives an $s$-colouring of $\Delta$.
If there exists an $s$-colouring, we say that 
$\Delta$ is \emph{s-colourable}.  
If $\Delta$ has dimension $d-1$, then we say that 
$\Delta$ is \emph{balanced} if it is $d$-colourable.
\end{definition}

\begin{example}  If $\Delta$ is simplicial complex on $|V| = n$ vertices,
then $\Delta$ is $n$-colourable; indeed, we take our colouring
to be $V = \{x_1\} \cup \{x_2\} \cup \cdots \cup \{x_n\}$.
\end{example}

\section{Vertex decomposable results}\label{vd results}

Starting with an $s$-colourable simplicial complex, we introduce a procedure
to construct a new simplicial complex that is pure of dimension $s-1$, balanced,
and furthermore, vertex decomposable.  The whiskering constructions
found in \cite{CN,V}  for flag complexes (equivalently, independence complexes
of graphs) are then special cases of our construction.

We build a new simplicial complex from $\Delta$ and a colouring of $\Delta$.

\begin{construction}\label{DeltaChi}  Let $\Delta$ be a simplicial complex on the vertex set $\{x_1, \dots, x_n\}$.  Given an $s$-colouring $\chi$ of $\Delta$ given by $V= V_1 \cup \dots \cup V_s$, we define  $\Delta_\chi$ on vertex set 
$\{x_1, \dots, x_n, y_1, \dots, y_s\}$ to be the simplicial complex with faces $\sigma \cup \tau$ where $\sigma$ is a face of $\Delta$ and $\tau$ is any subset of $\{y_1,\dots, y_s\}$ such that for all $y_j \in \tau$ we have $\sigma \cap V_j = \emptyset$. 
\end{construction}

\begin{example}\label{RunningExample}
Let $\Delta$ be the simplicial complex shown in Figure \ref{RunningExampleFig}. 
\begin{figure}[h!]
\begin{tikzpicture}[thick, scale =2]
\filldraw[fill= black!15!white] (0,0) node [anchor = north]{$x_1$} -- (.5, .866) -- (.5, .866) node [anchor = south]{$x_2$} -- (1,0) node [anchor = north]{$x_3$}--  cycle;
\draw (.5, .866) -- (1.5, .866) node [anchor = south]{$x_4$} -- (1,0);
\end{tikzpicture}
\caption{\label{RunningExampleFig} The simplicial complex with facets $\{x_1x_2x_3, x_2x_4, x_3x_4\}$.}
\end{figure}
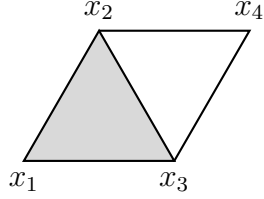
Let $\chi$ be the colouring of the vertices given by the partition $\{x_1, x_4\} \cup \{x_2\} \cup \{x_3\}$.  Then 
\[
 \Delta_\chi = \langle y_1y_2y_3, x_1y_2y_3, x_2y_1y_3, x_3y_1y_2, x_4y_2y_3, x_1x_2y_3, x_2x_3y_1, x_1x_3y_2, x_2x_4y_3, x_3x_4y_2, x_1x_2x_3\rangle.  
\]
\end{example}

\begin{remark} Observe that each $s$-colouring $\chi$ of $\Delta$ creates a new 
simplicial complex $\Delta_{\chi}$.  As we shall see, even though these simplicial
complexes $\Delta_{\chi}$ may be different, they all share some interesting properties,
regardless of how $\Delta$ is coloured.   
\end{remark}

\begin{remark} Construction \ref{DeltaChi} was recently introduced independently
by Frohmader \cite[Construction 7.1]{F}.  However, the construction appears 
in earlier work of Bj\"orner, Frankl, Stanley \cite{BFS} (e.g., see the proof in the 
Section 5 when when $a = (1,\ldots,1)$).  Another variation appears in work of Hetyei
(see \cite[Definition 4.2]{H}).
\end{remark}

We now prove some properties about our new complex $\Delta_{\chi}$.  

\begin{theorem} \label{correspondence}
The facets of $\Delta_\chi$ are in one-to-one 
correspondence with the faces of the original simplicial complex $\Delta$.  In addition $\Delta_\chi$ is pure of dimension $s-1$ and balanced.
\end{theorem}

\begin{proof}
Let $V = V_1\cup \cdots \cup V_s$ be the colouring $\Delta$ given by $\chi$.
It is clear from the definition of $\Delta_\chi$ that the maximal faces are those of the form $\sigma \cup \{y_j ~|~ V_j \cap \sigma = \emptyset\}$ where $\sigma$ is a face of $\Delta$.  This establishes the one-to-one correspondence between the faces of $\Delta$ and the facets of $\Delta_\chi$.  

If we partition the vertices of 
$\Delta_\chi$ as
\[
\{x_1, \dots, x_n, y_1, \dots, y_s\} = V_1' \cup V_2' \cup \dots \cup V_s'
\]
where $V_j' = V_j \cup \{y_j\}$, then this partition gives
an $s$-colouring of $\Delta_{\chi}$.   We can see from the characterization of the facets of $\Delta_\chi$ that each 
facet contains exactly one vertex from each 
of the sets $V_1', \dots, V_s'$, and hence $\Delta_\chi$ is pure of dimension $s-1$ as well as balanced.\end{proof}

\begin{example}
Let $\Delta = \langle x_1x_2x_3, x_2x_4, x_3x_4\rangle$ and $\chi$ the colouring given by $\{x_1, x_4\} \cup \{x_2\} \cup \{x_3\}$ (see Example \ref{RunningExample}).  The faces of $\Delta$ are 
\[
\Delta = \{\emptyset, x_1, x_2, x_3, x_4, x_1x_2, x_2x_3, x_1x_3, x_2x_4, x_3x_4, x_1x_2x_3\}.
\]  
These are in one-to-one correspondence with the facets of $\Delta_\chi$:
\begin{align*}
 \Delta_\chi = &\langle y_1y_2y_3, x_1y_2y_3, x_2y_1y_3, x_3y_1y_2, x_4y_2y_3, x_1x_2y_3, x_2x_3y_1, x_1x_3y_2, x_2x_4y_3, x_3x_4y_2, x_1x_2x_3\rangle.  
\end{align*}
\end{example}

We arrive at the main result of this section.

\begin{theorem} \label{vertex decomposable} 
For any simplicial complex $\Delta$, and any $s$-colouring $\chi$ of $\Delta$, 
the simplicial complex $\Delta_\chi$ is vertex decomposable.
\end{theorem}

\begin{proof}  Let $\Delta$ be a simplicial complex on the vertex set 
$\{x_1, \dots, x_n\}$.  We proceed by induction on $n$.  

If $\Delta$ is the simplicial complex consisting of a 
single vertex $x_1$, then the only possible colourings of the
vertices of $\Delta$ are of the form $V = V_1\cup \dots \cup V_s$ where $V_1 = \{x_1\}$ and $V_2, \dots, V_s$ are empty.  In this case $\Delta_\chi = \langle x_1y_2\dots y_s, y_1y_2\dots y_s\rangle$.  This is vertex decomposable since $\Delta_\chi \setminus x_1 = \langle y_1y_2\dots y_s\rangle$ and ${\rm link}_{\Delta_\chi}(x_1) = \langle y_2\dots y_s\rangle$ are both simplices.

Now suppose that $\Delta$ is a simplicial complex 
on the vertex set $V = \{x_1, \dots, x_n\}$, and let $\chi$ be 
the $s$-colouring of $\Delta$ given by $V = V_1 \cup \cdots \cup V_s$.  
We will  show that we can decompose $\Delta_\chi$ by decomposing at any 
vertex $x_i$.   Let $g_1,\dots, g_t$ be the faces of $\Delta$ and define 
$g_i' = \{y_j ~|~ V_j \cap g_i = \emptyset\}$.  
So $g_1 \cup g_1', \dots, g_t \cup g_t'$ are the facets of $\Delta_\chi$.

We must show that both $\Delta_\chi \setminus x_i$ and ${\rm link}_{\Delta_\chi}(x_i)$ are 
vertex decomposable.  First consider the deletion.  We may assume that the facets of 
$\Delta_\chi$ are ordered so that the facets $g_1 \cup g'_1, \dots, g_r \cup g'_r$ do not 
contain the vertex $x_i$ and the facets $g_{r+1}\cup g'_{r+1}, \dots, g_t \cup g'_t$ do 
contain $x_i$.  So
\[
\Delta \setminus x_i 
= \{\mbox{faces of } \Delta \mbox{ which do not contain } x_i\} = \{g_1, \dots, g_r\}.
\]
Note that we are using the fact that $g_1,\ldots,g_r,g_{r+1},
\ldots,g_t$ is a complete list of the faces of $\Delta$ by Theorem \ref{correspondence}.

Without loss of generality we may assume that $x_i \in V_1$.   Then 
$V \setminus \{x_i\} = (V_1 \setminus \{x_i\})  
\cup V_2 \cup \dots \cup V_s$ is an $s$-colouring of $\Delta \setminus x_i$.  Call this $s$-colouring $\chi'$. 
 Then
 $(\Delta \setminus x_i)_{\chi'} = \langle (g_1 \cup g'_1), 
\ldots, (g_r \cup g'_r)\rangle = \Delta_\chi \setminus x_i$. 
 Since $\Delta \setminus x_i$ is a simplicial complex on 
fewer than $n$ vertices,  $(\Delta \setminus x_i)_{\chi'}$ is vertex decomposable.

Now consider the link.  Since $(g_{r+1} \cup g'_{r+1}), \dots, (g_t \cup g'_t)$ 
are the facets of $\Delta_\chi$ which contain $x_i$, 
\begin{align*}
{\rm link}_{\Delta_\chi}(x_i) &= \langle (g_{r+1}\cup g'_{r+1}) \setminus \{x_i\}, \dots, (g_t \cup g'_t) \setminus \{x_i\} \rangle\\ 
&= \langle ((g_{r+1}\setminus \{x_i\}) \cup g'_{r+1}), \dots, ((g_t \setminus \{x_i\}) \cup g'_t) \rangle \, . 
\end{align*}
For each $1 \leq j \leq s$, set 
$W_j = \{x_\ell \in V_j ~|~ x_\ell \in {\rm link}_\Delta (x_i)\}$.  Note that some of these sets 
may be empty.  Then $W = W_1\cup \dots \cup W_s$ is an $s$-colouring of 
${\rm link}_{\Delta}(x_i)$.  We call this $s$-colouring $\chi''$.  Then
\[
({\rm link}_{\Delta}(x_i))_{\chi''} = {\rm link}_{\Delta_\chi}(x_i)
\]
and by induction $({\rm link}_{\Delta}(x_i))_{\chi''}$ is vertex decomposable.
\end{proof}

The fact that $\Delta_{\chi}$ is vertex decomposable has a number
of consequences.

\begin{definition} A pure $d$-dimensional simplicial complex $\Delta$ is \emph{shellable} if there is an ordering $F_1,\ldots,F_s$ on the facets of $\Delta$
such that for all $1 \leq i <  j \leq s$ there exists some $v \in F_j \setminus F_i$ and some $\ell \in \{1, \dots, j-1\}$ with $F_j \setminus F_\ell = \{v\}$.  Such an ordering on the facets is called a \emph{shelling order}.
\end{definition}

\begin{corollary}\label{shellingOrder} 
For any simplicial complex $\Delta$, and any $s$-colouring $\chi$ of $\Delta$,
the simplicial complex $\Delta_\chi$ is shellable, and thus, Cohen-Macaulay.
Moreover, any order of the facets of $\Delta_\chi$ which refines the 
order given by ordering the faces of $\Delta$ by increasing dimension is a shelling order.
\end{corollary}

\begin{proof}
By Theorem \ref{vertex decomposable}, $\Delta_\chi$ is vertex decomposable, so by \cite[Corollary 2.9]{BP} it is also shellable, and consequently, Cohen-Macaulay 
(e.g, see \cite[Theorem 8.2.6]{HH}). 

For the rest, let $F_1, \dots F_s$ be the facets of $\Delta_\chi$.  
By Theorem \ref{correspondence}, each $F_i = g_i \cup g'_i$  
where $g_i$ is a face of $\Delta$ and 
$g'_i = \{y_j ~|~ V_j \cap g_i = \emptyset\}$.  
We order the facets $F_1, \dots F_s$ so that $\dim g_i \leq \dim g_j$ if $i <j$.  We
now show that this is a shelling order.

Let $F_i$, $F_j$ be any two distinct facets 
of $\Delta_\chi$ with $i <j$.  Since $i<j$, we have $\dim g_i \leq \dim g_j$ and so there is some $x_u \in F_j \setminus F_i$.  Since $g_j \setminus \{x_u\}$ is a face of $\Delta$ we have $g_j \setminus \{x_u\} = g_\ell$ for some $\ell$, and since $\dim g_\ell < \dim g_j$ we have $\ell < j$.
Since $g_\ell = g_j \setminus \{x_u\}$ we must have $g'_\ell = g'_j \cup \{y_w\}$ where $x_u \in V_w$.  Then 
\[F_j \setminus F_\ell = (g_j \cup g'_j) \setminus (g_\ell \cup g'_\ell) = (g_j \cup g'_j) \setminus ((g_j \setminus \{x_u\}) \cup (g'_j \cup \{y_w\})) = \{x_u\} \, .\]
Thus our ordering is a shelling order.
\end{proof}

Results of Villarreal \cite{V}, Dochtermann and Engstr\"om \cite{DE}, and
Cook and Nagel \cite{CN} on the independence complexes of graphs now become
special cases of Theorem \ref{vertex decomposable}.   We first recall the relevant
terminology.  This terminology will be also be used in Section 5.

Let $G = (V_G,E_G)$ be a finite simple graph on the vertex set
$V_G = \{x_1,\ldots,x_n\}$ and edge set $E_G$.  One can use the independent
sets of $G$ to define a simplicial complex.

\begin{definition}
A subset $W \subseteq V_G$ is an {\it independent set} of a graph $G$ if for
every edge $e \in E_G$, we have $e \not\subseteq W$.  A set $W$ is a {\it maximal
independent} set if $W$ is an independent set, but is not a proper
subset of any other independent set of $G$.
\end{definition}

\begin{definition} Let $G$ be a graph.  The {\it independence complex} of $G$, denoted
${\rm Ind}(G)$, is the simplicial complex defined by 
\[{\rm Ind}(G) = \{W \subseteq V_G ~|~ \mbox{$W$ is an independent set of $G$}\}.\]
The independence complex ${\rm Ind}(G)$ is sometimes called a {\it flag complex}.
\end{definition}

\begin{definition} The {\it clique} of order $n$, denoted
$K_n$, is the graph with vertex set $V = \{x_1,\ldots,x_n\}$
and edge set $\{ \{x_i,x_j\} ~|~ 1 \leq i < j \leq n\}$.  Note that an isolated vertex
can be viewed as $K_1$.
\end{definition}

Given any graph $G$ and any subset $S \subseteq V_G$, 
the {\it induced graph on $S$}, denoted $G|_S$, is the graph with
vertex set $S$ and edge set $E_{G|_s} = \{e \in E_G ~|~ e \subseteq S\}$.

\begin{definition}
Let $G = (V_G,E_G)$ be a finite simple graph.  A {\it clique partition} of $V_G$
is a partition of $V_G = V_1 \cup V_2 \cup \cdots \cup V_s$ such that
each induced graph $G|_{V_i}$ is a clique.
\end{definition}

\begin{construction}[Cook-Nagel]\label{cliquewhisker}  Let $\pi$ denote a clique
partition $V_G = V_1 \cup \cdots \cup V_s$ for a finite simple graph
$G = (V_G,E_G)$.  From $G$ and $\pi$, let $G^{\pi}$ denote the
finite simple graph on the vertex set $V_{G^{\pi}} = V_G \cup \{y_1,\ldots,y_s\}$
and edge set 
\[E_{G^{\pi}} = E_G \cup \bigcup_{i=1}^{s} \{ \{x,y_i\} ~|~ x \in V_i\}.\]
In other words, add a new vertex for each partition $V_i$, and join this
new vertex to every vertex in $V_i$.  We call $G^{\pi}$ a \emph{clique whiskering}
of $G$.
\end{construction}

\begin{corollary}[{\cite[Theorem 3.3]{CN}}]
Let $G$ be a graph, and let $G^{\pi}$ denote the clique-whiskered graph.
Then ${\rm Ind}(G^{\pi})$ is vertex decomposable. 
\end{corollary}

\begin{proof}  Let $\pi$ be the clique partition $V_G = V_1 \cup \cdots \cup V_s$.
Then the faces of ${\rm Ind}(G^{\pi})$ have the form $\sigma \cup \tau$ where 
$\sigma$ is an independent set of $G$, i.e. $\sigma \in {\rm Ind}(G)$, $\tau$
is a subset of $\{y_1,\ldots,y_s\}$, and if
 $y_j \in \tau$, then $\sigma \cap V_j = \emptyset$.
It now suffices to note that $\pi$ is also an $s$-colouring of ${\rm Ind}(G)$,
from which it will follow that
${\rm Ind}(G^{\pi}) = {\rm Ind}(G)_{\pi}$.  Indeed, for any facet
$F \in {\rm Ind}(G)$, we must have $|F \cap V_i| \leq 1$ since $F$ is an
independent set but all vertices of $V_i$ are adjacent since $G|_{V_i}$ is a clique. 
\end{proof}

\begin{remark}
Villarreal \cite{V} first introduced Construction \ref{cliquewhisker} in the special case that
the partition $\pi$ was $V_G = \{x_1\} \cup \{x_2\} \cup \cdots \cup \{x_n\}$.
For this partition $\pi$,  it was shown in 
\cite[Theorem 4.4]{DE} that ${\rm Ind}(G^{\pi})$
was vertex decomposable.
\end{remark}


\section{$h$-vectors and algebraic consequences}\label{algebra}

In this section, we explore some consequences
of Theorem \ref{vertex decomposable}.  In particular, we show
that any $f$-vector of a simplicial complex is also the $h$-vector
of a balanced, vertex decomposable simplicial complex.
This enables us to give a new characterization of $f$-vectors
of simplicial complexes, which extends Bj\"orner, Frankl and Stanley's \cite{BFS}
characterization.  We also show 
that for any $f$-vector $f(\Delta)$, there exists a
square-free monomial ideal
with a linear resolution whose graded Betti numbers are a function 
of $f(\Delta)$.
We relate this idea to recent work of Herzog,
Sharifan, and Varbaro \cite{HSV}.

We begin by recalling the definition of an $h$-vector.

\begin{definition} The $h$-vector $(h_0, h_1, \dots, h_{d+1})$ of a $d$-dimensional 
simplicial complex $\Delta$, denoted $h(\Delta)$
is defined in terms of the $f$-vector $f(\Delta) = (f_{-1},f_0,\ldots,f_d)$ as follows
\[
h_k = \sum_{i = 0}^k (-1)^{k-i}{d-i \choose k-i}f_{i-1}(\Delta) \, .
\]
\end{definition}

We can use Corollary \ref{shellingOrder} to give a proof of the following
result:

\begin{theorem}  \label{hvector}
The following containment of sets holds:
\[\left\{
\begin{array}{c}
\mbox{$f$-vectors of } \\
\mbox{simplicial complexes}
\end{array}
\right\} \subseteq
\left\{
\begin{array}{c}
\mbox{$h$-vectors of balanced}\\
\mbox{vertex decomposable simplicial complexes}
\end{array}\right\}.\]
\end{theorem}

\begin{proof}
Let $f(\Delta)$ be the $f$-vector of a simplicial complex $\Delta$.
For any $s$-colouring $\chi$ of $\Delta$, 
$\Delta_\chi$ is a balanced vertex decomposable simplicial complex
by Theorems \ref{correspondence} and \ref{vertex decomposable}, and
thus shellable.  We will show the $h$-vector of $\Delta_{\chi}$ is $f(\Delta)$.

If $F_1,\ldots, F_s$ are the facets of $\Delta_{\chi}$, then
by Theorem \ref{correspondence}, each $F_i = g_i \cup g'_i$  
where $g_i$ is a face of $\Delta$ and 
$g'_i = \{y_j ~|~ V_j \cap g_i = \emptyset\}$.  Moreover, by Corollary
\ref{shellingOrder}, we have a shelling if 
we order the facets $F_1, \dots F_s$ so that $\dim g_i \leq \dim g_j$ if $i <j$. 

Because we have a shelling, \cite[Proposition 8.2.7]{HH} allows us to 
construct the following partition of $\Delta_{\chi}$:
\[\Delta_{\chi} = \bigcup_{i=1}^s[\mathcal{R}(F_i),F_i].\]
Here, $[G,F]$ is an interval, i.e., $[G,F] = \{H \in \Delta_{\chi} ~|~ G \subseteq H \subseteq F\}$
and 
\[\mathcal{R}(F_i) = \{z \in F_i ~|~ F_i \setminus \{z\} \in \langle F_1,\ldots,F_{i-1} \rangle\}
.\]
By \cite[Proposition 2.3]{St}, the $h$-vector of $\Delta_{\chi}$ satisfies
\[h_i = |\{j ~|~ |\mathcal{R}(F_j)| = i \}|~~\mbox{for $i=0,\ldots,d+1$.}\]
The conclusion now follows from the fact that $\mathcal{R}(F_j) = g_j$, so $h_i$ counts the
number of faces of dimension $i-1$ in $\Delta$, whence 
$h(\Delta_{\chi}) =(h_0,\ldots,h_{d+1}) = (f_{-1},f_1,\ldots,f_d) = f(\Delta)$.
\end{proof}

Theorem \ref{hvector} allows us to add a new equivalent statement
to a theorem of Bj\"orner, Frankl, and Stanley \cite{BFS}.

\begin{theorem} \label{BFSThem}
Let $m = (m_1,\ldots,m_t) \in \mathbb{Z}_+^t$.  Then
the following are equivalent:
\begin{enumerate}
\item[$(i)$] $m$ is the $f$-vector of a simplicial complex.
\item[$(ii)$] $m$ is the $h$-vector of a balanced, vertex decomposable simplicial complex.
\item[$(iii)$] $m$ is the $h$-vector of a balanced, shellable simplicial complex.
\item[$(iv)$] $m$ is the $h$-vector of a balanced, Cohen-Macaulay simplicial complex.
\end{enumerate}
\end{theorem}

\begin{proof}
Theorem \ref{hvector} gives $(i) \Rightarrow (ii)$.  The statements
$(ii) \Rightarrow (iii)$ and $(iii) \Rightarrow (iv)$ follow from
the implications:
\[\mbox{vertex decomposable} ~~ \Rightarrow \mbox{shellable} ~~ \Rightarrow 
\mbox{Cohen-Macaulay}.\]
Finally, $(iv) \Rightarrow (i)$ was first proved by Stanley \cite{S}.  
The equivalence of $(i)$, $(iii)$ and $(iv)$ were first shown
in \cite{BFS}, albeit in a much more general setting.
\end{proof}

It is natural to ask if Theorem \ref{BFSThem} still holds if we restrict
to smaller classes of simplicial complexes.  For example, 
it has been asked whether the above statements still hold if we 
replace an arbitrary simplicial complex with the class of flag complexes.
In particular,
Cook and Nagel \cite{CN}, and Constantinescu and Varbaro \cite{CV} have posited the 
following conjecture (the conjecture of Cook and Nagel does not include
the word balanced):

\begin{conjecture}  \label{hvectorconjecture}
The following equality of sets holds:
\[\left\{
\begin{array}{c}
\mbox{$f$-vectors of } \\
\mbox{flag complexes}
\end{array}
\right\} =
\left\{
\begin{array}{c}
\mbox{$h$-vectors of balanced}\\
\mbox{vertex decomposable flag complexes}
\end{array}\right\}\]
\end{conjecture}

One can show that the containment of Theorem \ref{hvector} still holds
true for flag complexes.  We omit the proof here, but instead point the
reader to the proofs of \cite[Corollary 3.10]{CN}
and \cite[Proposition 4.1]{CV}.    The second proof is interesting
since the authors use basically the same construction as 
Construction \ref{DeltaChi}, but in the special case that
the colouring is given by the partition $V = \{x_1\} \cup \cdots \cup \{x_n\}$.  In some special
cases, e.g., bipartite graphs (see \cite{CV}),
the conjecture has been proved.  We add additional evidence for Conjecture
\ref{hvectorconjecture} when we prove the statement for the flag complexes
of chordal graphs in the next section.

We conclude this section by showing how to use Theorem \ref{hvector}
to find the graded Betti numbers of the Alexander dual of the Stanley-Reisner
ideal associated to $I_{\Delta_{\chi}}$.  
The
following well-known definition connects
simplicial complexes and monomial ideals.

\begin{definition} Given a simplicial complex $\Delta$ on the vertex set
 $\{x_1, \dots, x_n\}$, the \emph{Stanley-Reisner ideal} of $\Delta$ 
is the monomial ideal
\[
I_\Delta = (x_{i_1}x_{i_2}\cdots x_{i_s} ~|~ \{x_{i_1}, x_{i_2}, \dots, x_{i_s}\} \notin \Delta)
\]
in the ring $S = k[x_1, \dots, x_n]$.  The quotient ring $S/I_\Delta$ is the \emph{Stanley-Reisner ring}.  
\end{definition}

For completeness, we also recall
the definition of the Alexander dual.  

\begin{definition}
Given a subset $\sigma \subseteq \{x_1,
\ldots,x_n\}$, let $\overline{\sigma} = \{x_1,\ldots,x_n\} \setminus \sigma$.
The {\it Alexander dual} of a simplicial complex $\Delta$, denoted
$\Delta^{\vee}$, is the simplicial complex
$\Delta^{\vee} = \{\overline{\sigma} ~|~ \sigma \not\in \Delta\}.$
\end{definition}

\begin{theorem} \label{bettinumbers}
Let $(f_{-1},f_0,\ldots,f_{d})$ be the $f$-vector
of a $d$-dimensional simplicial complex $\Delta$
on $V= \{x_1,\ldots,x_n\}$, and let $\chi$
be any $s$-colouring of $\Delta$.   The graded Betti numbers
of $I_{\Delta_{\chi}^\vee}$  in $R$ are given by the formula
\[
\beta_{i, i+n} (I_{\Delta_\chi^\vee}) = \sum_{j = i}^{d+1} {j \choose i}f_{j-1}(\Delta).
\]
In particular,
$\operatorname{proj-dim}(I_{\Delta_\chi^\vee}) = \operatorname{reg}(R/I_{\Delta_{\chi}})=d+1.$
\end{theorem}
\begin{proof}  
The projective dimension follows directly from our formula, and for
the regularity, we use the identity (e.g., see \cite[Proposition 8.1.10]{HH})
that $\operatorname{proj-dim}(I_{\Delta^{\vee}}) = \operatorname{reg}(R/I_{\Delta})$.

Because $\Delta_{\chi}$ is pure and vertex decomposable (and thus shellable),
\cite[Corollary 5]{ER} gives
\begin{equation}\label{erform}
\sum_{i \geq 1} \beta_i(R/I_{\Delta_{\chi}^\vee})t^{i-1} = \sum_{i \geq 0}
h_i(\Delta_{\chi})(t+1)^i.
\end{equation}
Note that in \cite{ER}, the authors are taking the resolution of $R/I_{\Delta_{\chi}^{\vee}}$,
so $\beta_i(R/I_{\Delta_{\chi}^\vee}) = \beta_{i-1}(I_{\Delta_{\chi}^{\vee}})$.  
Furthermore, although the formula of \cite{ER}
is expressed in terms of total
graded Betti numbers, the resolution of $I_{\Delta_\chi^{\vee}}$ is linear (this is because
$\Delta_{\chi}$ is shellable and pure of dimension $s-1$, and
 hence $I_{\Delta_{\chi}^{\vee}}$ is generated in degree $n$ and is componentwise linear,
which implies the ideal has a linear resolution).  We therefore have
 $\beta_{i-1}(I_{\Delta_{\chi}^{\vee}}) = \beta_{i-1,n+i-1}(I_{\Delta_{\chi}^{\vee}})$.

To finish the proof,  Theorem \ref{hvector} allows us to replace $h_i(\Delta_{\chi})$ with
$f_{i-1}(\Delta)$  in the formula \eqref{erform}, thus giving the desired formula for
 of $\beta_{i-1,n+i-1}(I_{\Delta_{\chi}^{\vee}})$.
\end{proof}

\begin{example}
Let $\Delta$ be the simplicial complex of Example \ref{RunningExample} and let
$\chi$ be the 3-colouring given by the partition $\{x_1, x_4\} \cup\{x_2\} \cup \{x_3\}$.  The $f$-vector of $\Delta$ is
\[
f(\Delta) = (1, 4, 5, 1).
\]
Then applying the formula from Theorem \ref{bettinumbers}, we see that the Betti numbers of $I_{\Delta_\chi^\vee}$ are
\[
\beta_{0, 4}(I_{\Delta_\chi^\vee}) = 11, ~~ \beta_{1, 5}(I_{\Delta_\chi^\vee}) = 17,  ~~ \beta_{2, 6}(I_{\Delta_\chi^\vee}) = 8, ~~ \beta_{3, 7}(I_{\Delta_\chi^\vee}) = 1.
\]
\end{example}

\begin{remark}
For any valid $f$-vector $f(\Delta) = (f_0,\ldots,
f_d)$, the sequence
\begin{equation}\label{bettisequence}
\left(\sum_{j=0}^{d+1} \binom{j}{0} f_{j-1}(\Delta),\sum_{j=1}^{d+1} \binom{j}{1} f_{j-1}(\Delta),\ldots,\sum_{j=d+1}^{d+1} \binom{j}{d+1} f_{j-1}(\Delta)\right)
\end{equation}
is a valid sequence of Betti numbers for an ideal with a linear resolution
by Theorem \ref{bettinumbers}.
Herzog, Sharifan, and Varbaro \cite{HSV} classified all
valid sequences of Betti numbers for an ideal with a linear resolution.  In
particular, they proved that $m = (m_0,m_1,\ldots,m_{d+1})$ is an $O$-sequence
if and only if  
\[\left(\sum_{j=0}^{d+1} \binom{j}{0} m_{j},\sum_{j=1}^{d+1} \binom{j}{1} m_{j},\ldots,\sum_{j=d+1}^{d+1} \binom{j}{d+1} m_j \right)\]
is a valid sequence of Betti numbers for an ideal with a linear resolution.

Because $f$-vectors are $O$-sequences, \cite{HSV} also implies 
that \eqref{bettisequence} is the Betti sequence of an ideal with linear resolution.
Our work, in particular Theorem \ref{bettinumbers}, highlights how to start
with a simplicial complex with a given $f$-vector, and find a square-free
monomial ideal whose graded linear resolution has Betti sequence given
by \eqref{bettisequence}.  This contrasts with the main results of
\cite{HSV} since the
ideal they construct with Betti sequence \eqref{bettisequence} 
need not be a square-free monomial
ideal. 
\end{remark}


\section{Application: independence complexes of chordal graphs}

In the previous section, we saw how to construct a balanced, vertex
decomposable simplicial complex $\Delta_\chi$ from any simplicial complex 
$\Delta$ and any $s$-colouring $\chi$ of $\Delta$
with the property that $f(\Delta) = h(\Delta_\chi)$.  In this section,
 we give a criterion
for when this construction can be reversed.  As an application, 
we study $h$-vectors and $f$-vectors of 
the independence complexes of chordal graphs.

We start with our criterion for ``reversing'' the process of the last section.

\begin{definition}  
Let $\Delta$ be a simplicial complex on the vertex set $V$ and let $W \subseteq V$.  The {\em restriction of $\Delta$ to $W$} is the subcomplex
\[
\Delta |_W = \{F \in \Delta ~|~ F \subseteq W\}.
\]
\end{definition}

\begin{definition}
Suppose $\Delta = \langle F_1, \dots, F_s\rangle$ is a simplicial complex on the vertex set $V$.  We say that $\Delta$ has a \emph{facet restriction with respect to $F$} if $F$ is a facet of $\Delta$ such that 
\[
\Delta |_{V \setminus F} = \{F_1 \setminus F, \dots, F_s \setminus F\}.
\]
\end{definition}

Note that the inclusion 
$\Delta |_{V \setminus F} \supseteq 
\{F_1 \setminus F, \dots, F_s \setminus F\}$ always holds; however, 
in general the two sets may not be equal as we see in the following example.

\begin{example}
Let $\Delta = \langle 123, 234, 345, 456 \rangle$ (see Figure \ref{no facet restriction}).
By considering each facet of $\Delta$, we can show it has no facet restriction.
Let $F$ be the facet $123$.  Then
\[
\Delta |_{V \setminus F} =
 \Delta |_{456} =
 \{ \emptyset, 4, 5, 6, 45, 56, 46, 456\} 
\neq \{123 \setminus F, 234 \setminus F, 345 \setminus F, 456 \setminus F\}
 = \{\emptyset, 4, 45, 456\}\, .
\]
Similarly, if we consider the facet $234$ we see that
\[
\Delta |_{V \setminus 234} = \Delta |_{156} 
= \{ \emptyset, 1, 5, 6, 56\}
 \neq \{123 \setminus 234, 234 \setminus 234, 345 \setminus 234, 456 \setminus 234\} 
= \{1, \emptyset, 5, 56\}\, .
\]
By symmetry we also have
\[
\Delta |_{V \setminus 345}
 \neq \{123 \setminus 345, 234 \setminus 345, 345 \setminus 345, 456 \setminus 345\}
\]
and 
\[
\Delta |_{V \setminus 456}
 \neq \{123 \setminus 456, 234 \setminus 456, 345 \setminus 456, 456 \setminus 456\} \, .
\]
Therefore the simplicial complex $\Delta$ has no facet restriction.
\end{example}

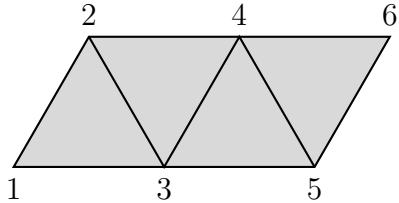
\begin{figure}[h]
\begin{tikzpicture}[thick, scale =2]
\filldraw[fill= black!15!white] (0,0) node [anchor = north]{1} -- (.5, .866) -- (2.5, .866) node [anchor = south]{6} -- (2,0) --  cycle;
\draw (.5, .866) node [anchor = south]{2}-- (1,0) node [anchor = north]{3} -- (1.5, .866) node [anchor = south]{4} -- (2,0) node [anchor = north]{5};
\end{tikzpicture}
\caption{\label{no facet restriction} The simplicial complex $\langle 123, 234, 345, 456 \rangle$ has no facet restriction.}
\end{figure}

\begin{example}  \label{facet restriction example}
Let $\Delta$ be the simplicial complex 
$\langle124, 245, 235, 456 \rangle$ (see Figure \ref{facet restriction}).
 Then $\Delta$ has a facet restriction with respect to the facet $245$ since 
\[
\Delta |_{V \setminus 245}  = \Delta |_{136} = \{\emptyset, 1, 3, 6\}
 =  \{124 \setminus 245, 245 \setminus 245, 235 \setminus 245, 456 \setminus 245\} \, .
\]
\end{example}

\begin{figure}[h]
\begin{tikzpicture}[thick, scale = 2]
\filldraw[fill= black!15!white](0,0) node [anchor = south]{1} -- (2,0) node [anchor = south]{3} -- (1, -1.732) node [anchor = north]{6} -- cycle;
\draw(1,0) node [anchor = south]{2} -- (1.5, -.866) node [anchor = north west]{5} -- (.5, -.866) node [anchor = north east]{4}-- cycle;
\end{tikzpicture}
\caption{\label{facet restriction} The simplicial complex $\langle 124, 245, 235, 456\rangle$ has a facet restriction with respect to the facet 245.}
\end{figure}
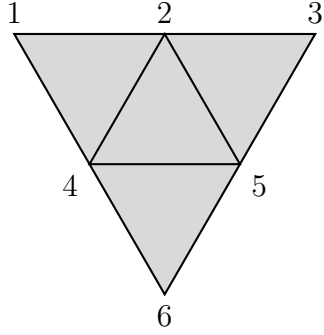

The existence of a facet restriction allows us to find a converse to 
Theorem \ref{hvector}.

\begin{theorem}\label{criterion}
Let $\Delta = \langle F_1,\ldots,F_t \rangle$ be a pure, balanced
simplicial complex such that $\Delta$ has a 
facet restriction with respect to the facet $F$.  
Then $\Delta = (\Delta|_{V\setminus F})_\chi$ where $\chi$ is
the colouring induced from the colouring of $\Delta$.  In
particular,
 $\Delta$ is vertex decomposable and $h(\Delta) = f(\Delta |_{V \setminus F})$. 
\end{theorem}

\begin{proof}
Let $d-1$ be the dimension of $\Delta$.  
Because $\Delta$ is pure and balanced, the colouring $\chi$ is given by a partition
$V = V_1 \cup V_2 \cup \cdots \cup V_d$ such that 
$|F_j \cap V_i| = 1 ~~\mbox{for all $1 \leq j \leq t$ and $1 \leq i \leq d$}.$
After relabelling, we can assume that $F_1$ is the facet that gives that
facet restriction.
Note that $\Delta|_{V\setminus F_1}$ is a simplicial complex on $Y = V \setminus F_1$, and
 is $d$-colourable since  $\Delta|_{V\setminus F_1}$ inherits a colouring from $\chi$
given by:
\[ Y = V\setminus F_1 = (V_1 \setminus F_1) \cup (V_2 \setminus F_1)
\cup \cdots \cup (V_d \setminus F_1).\]
Abusing notation, let $\chi$ denote this new colouring.
Then $(\Delta|_{V\setminus F_1})_\chi$ 
is a balanced vertex decomposable simplicial complex 
such that $h((\Delta|_{V\setminus F_1})_\chi) = f(\Delta|_{V\setminus F_1})$ by 
Theorem \ref{hvector}.

To complete the proof, it suffices to show that $(\Delta|_{V\setminus F_1})_\chi$ and $\Delta$
are the same simplicial complexes, but with a different labelling of the vertices.
By Theorem \ref{correspondence}, the facets of $(\Delta|_{V\setminus F_1})_\chi$ are in one-to-one
correspondence with the faces of $\Delta|_{V\setminus F_1}$.  But we also
have that the facets of $\Delta$ are in one-to-one correspondence with the
faces of $\Delta|_{V\setminus F_1}$ via the map $F_i \mapsto F_i \setminus F_1$.  
Indeed, this map is clearly onto by our assumption that $\Delta$ has a facet restriction with respect to $F_1$.
It suffices to show that this map is one-to-one.  So, suppose
$F_i \setminus F_1 = F_j \setminus F_1$, but $F_i \neq F_j$.  This
means that there is a vertex $x \in F_i \setminus F_j$ because the simplicial
complex is pure.  Since
 $\Delta$ is balanced, there is a vertex $y \in F_j \setminus
F_i$ with the same colour as $x$.  Because $F_i \setminus F_1 = F_j\setminus F_1$,
we must have $x$ and $y$ in $F_1$.  But this contradicts the colouring of $\Delta$.
By combining these two one-to-one correspondences, we get the desired
bijection between the facets of $\Delta$ and $(\Delta|_{V\setminus F_1})_\chi$.  
\end{proof}

\begin{example}
In Example \ref{facet restriction example} we saw that the simplicial complex 
\[
\Delta = \langle 124, 245, 235, 456 \rangle
\]
has a facet restriction with respect to the facet 245.  
Since $\Delta |_{136} = \{\emptyset, 1, 3, 6\}$,  the $f$-vector of $\Delta|_{136}$ is $f(\Delta |_{136}) = (1, 3)$.  Therefore the $h$-vector of $\Delta$ is 
$
h(\Delta) = f(\Delta|_{136}) = (1, 3).$
\end{example}

To apply Theorem \ref{criterion} to a class of simplicial complexes,
we need to justify the existence of facet restrictions.  We round
out this paper by focusing on the independence complexes (as introduced
in Section 3) of chordal graphs.  We first recall:

\begin{definition}
A graph $G$ is {\it chordal} if every induced cycle of $G$ of length $\geq 4$ has
a chord.
\end{definition}

We will prove the following fact about the independence complexes of chordal graphs. 

\begin{lemma}\label{chordal facet restriction}
Let $\Delta = {\rm Ind}(G)$ be the independence complex of a chordal
graph $G$.  If $\Delta$ is also pure, then $\Delta$ has a facet restriction.
\end{lemma}

From this lemma, we can deduce the following result, which proves a special case of Conjecture \ref{hvectorconjecture}.

\begin{theorem}
We have the 
the following equivalence of sets:
\[\left\{
\begin{array}{c}
\mbox{$f$-vectors of independence} \\
\mbox{complexes of chordal graphs}
\end{array}
\right\} =
\left\{
\begin{array}{c}
\mbox{$h$-vectors of balanced, vertex decomposable}\\
\mbox{independence complexes of chordal graphs}
\end{array}\right\}.\]
\end{theorem}

\begin{proof}
If $f(\Delta)$ is the $f$-vector of $\Delta = {\rm Ind}(G)$ 
when $G$ is chordal, then for any 
colouring $\chi$ of $\Delta = {\rm Ind}(G)$, the simplicial
complex $\Delta_{\chi}$ is balanced and vertex decomposable
by Theorem \ref{vertex decomposable}, 
and $f(\Delta) = h(\Delta_{\chi})$ by Theorem \ref{hvector}.  It remains
to note that $\Delta_{\chi}$ is the independence complex of the graph $G^{\chi}$,
the clique whiskering of $G$
using the clique partition of $G$ induced by $\chi$.   Furthermore,
it follows from Construction \ref{cliquewhisker} that if $G$ is chordal, then
so is $G^{\chi}$.  This completes the first containment.

To show the reverse containment, let $G$ be any chordal graph such that
$\Delta = {\rm Ind}(G)$ is balanced and vertex decomposable.  
Because $\Delta$ is vertex decomposable, and thus pure,
by Lemma \ref{chordal facet restriction}, the simplicial
complex $\Delta$ has a facet restriction with respect to some facet $F$.
But then by Theorem \ref{criterion}, we have $h(\Delta) = f(\Delta|_{V \setminus F})$.  
To complete the argument, we note that  
\[
\Delta|_{V \setminus F} = {\rm Ind}(G)|_{V \setminus F} 
= {\rm Ind}(G |_{V \setminus F}).\]  
The graph $G |_{V \setminus F}$ is an induced subgraph of a chordal graph, and so is a chordal graph.  So $h(\Delta) = f({\rm Ind}(G|_{V \setminus F}))$, thus completing the proof.
\end{proof}

\begin{remark}  We note that if a chordal graph has a pure independence complex then that independence complex is always vertex decomposable and balanced, so the right-hand side of the above corollary could be the set of $h$-vectors of pure independence complexes of chordal graphs.
\end{remark}

To prove Lemma \ref{chordal facet restriction} we will require a 
result of Herzog, Hibi, and Zheng.  We first describe another simplicial complex one
can associate to a graph.

\begin{definition}  
For any finite simple graph $G = (V_G,E_G)$ the \emph{clique complex} of $G$ is the simplicial complex 
\[Cl(G) = \{C\subseteq V ~|~ G |_C ~~\mbox{is a clique}\} \, .\]
\end{definition}

\begin{definition}  
Let $\Delta$ be a simplicial complex with vertex set $V$.  
We call $v \in V$ a \emph{free vertex} if $v$ is contained in exactly one facet of $\Delta$.
\end{definition}

\begin{theorem}[{\cite[Theorem 2.1]{HHZ}}]\label{HHZ} Let $G$ be a chordal graph and let $C_1, \dots, C_t$ be all the facets of $Cl(G)$ that contain a free vertex.  The following are equivalent:
\begin{enumerate}
\item[$(a)$] $R/I_{{\rm Ind}(G)}$ is Cohen-Macaulay.
\item[$(b)$] $G$ is unmixed, i.e., all maximal independent sets have the same cardinality.
\item[$(c)$] $V = C_1 \cup C_2 \cup \dots \cup C_t$ is a partition of the vertices of $G$.
\end{enumerate}
\end{theorem}

We are now ready to prove Lemma \ref{chordal facet restriction}.

\begin{proof} (of Lemma \ref{chordal facet restriction}) Let $\Delta = {\rm Ind}(G)$
be the independence complex of a chordal graph, and furthermore, assume $\Delta$ is 
pure.   Let $C_1, \dots, C_t$ be the facets of $Cl(G)$ which contain a free vertex.  
Since $\Delta$ is pure, we know that $G$ is unmixed.  Thus by Theorem \ref{HHZ},
we have the partition $V = C_1 \cup \dots \cup C_t$.  
For $1 \leq i \leq t$, let $y_i$ be a free vertex of $Cl(G)$ contained in $C_i$.  
Set $F = \{y_1, \dots, y_t\}$.  We will show that $F$ is a facet of $\Delta$ and that  
$\Delta$ has a facet restriction with respect to $F$. 

It is clear that $F = \{y_1, \dots, y_t\}$ is an independent set since each $y_i$ is in a unique maximal clique $C_i$ and an edge $\{y_i, y_j\}$ would constitute a clique of size 2.   Further, $F$ is a maximal independent set since every vertex $x \notin F$ is in some $C_i$ and therefore adjacent to $y_i$.  Since $\Delta$ is assumed to be pure, this means that every facet has size $t$.
 
 Finally, let $F_1, \dots, F_s$ be the facets of $\Delta$.  To finish the proof 
we will show that 
 \[  \Delta |_{V \setminus F} = \{F_1 \setminus F, \dots, F_s \setminus F\} \, .\]
We simply need to
show $\Delta|_{V \setminus F} \subseteq \{F_1 \setminus F, \dots, F_s \setminus F\}$.
Let $H \in \Delta|_{V \setminus F}$, and define $H' = H \cup \{y_i ~|~ C_i \cap H = \emptyset\}$.  Then $H'$ is independent since the neighbours of $y_i$ are the elements of $C_i \setminus \{y_i\}$.  Since $H'$ has cardinality $t$, it is a facet of $\Delta$.  Therefore $H = H' \setminus F$ which proves that $\Delta |_{V \setminus F} = \{F_1 \setminus F, \dots, F_s \setminus F\}$.
\end{proof}


\end{document}